\documentclass[12pt,a4paper]{article}
\usepackage[affil-it]{authblk}
\usepackage[cm]{fullpage}
\usepackage{comment}
\usepackage{natbib}
\usepackage{xcolor}
\usepackage{graphicx}
\usepackage{tabularx}
\usepackage{threeparttable} 
\usepackage{etoolbox} 
\appto\TPTnoteSettings{\footnotesize} 
\usepackage[hidelinks]{hyperref}
\usepackage{multirow}

\usepackage{algorithm}
\usepackage{algpseudocode}

\algnewcommand{\algorithmicgoto}{\textbf{go to}}%
\algnewcommand{\Goto}[1]{\algorithmicgoto~\ref{#1}}%

\usepackage{capt-of} 
\usepackage{amsthm}
\theoremstyle{plain}
\newtheorem{theorem}{Theorem}
\newtheorem{lemma}{Lemma}

\theoremstyle{definition}

\newtheorem{proposition}{Proposition}

\usepackage{booktabs}
\usepackage{amsmath}
\usepackage{amssymb}
\usepackage{bm} 
\usepackage{mathrsfs} 
\usepackage{bbm}
\usepackage{tikz}
\usetikzlibrary{shapes,snakes} 
\usetikzlibrary{arrows.meta}
\newcounter{counter}

\definecolor{GRAY}{gray}{0.97}
\usepackage{mathtools} 
\allowdisplaybreaks
\usepackage[symbol*]{footmisc}
\DefineFNsymbolsTM{myfnsymbols}{
	\textasteriskcentered *
}%
\setfnsymbol{myfnsymbols}
\providecommand{\keywords}[1]{\textbf{\textit{Keywords---}} #1}
\newcommand\floor[1]{\displaystyle\lfloor#1\rfloor}
\newcommand\ceil[1]{\lceil #1 \rceil}
\newcommand\bfloor[1]{\displaystyle\left\lfloor#1\right\rfloor}
\newcommand\bceil[1]{\displaystyle\left\lceil #1 \right\rceil}
\newcommand\fdiv[2]{\lfloor\frac{#1}{#2}\rfloor}
\newcommand\cdiv[2]{\lceil \frac{#1}{#2}\rceil}
\newcommand\bfdiv[2]{\displaystyle\left\lfloor\dfrac{#1}{#2}\right\rfloor}
\newcommand\bcdiv[2]{\displaystyle\left\lceil \dfrac{#1}{#2}\right\rceil}
\usepackage{ragged2e}

\title{An analytical bound on the fleet size in vehicle routing problems: a dynamic programming approach}

\author[1]{Ali Eshragh}
\author[1]{Rasul Esmaeilbeigi\thanks{E-mail address: \textit{rasul.esmaeilbeigi@uon.edu.au}}}
\author[2]{Richard Middleton}
\affil[1]{School of Mathematical and Physical Sciences, The University of Newcastle, Australia}
\affil[2]{School of Electrical Engineering and Computing, The University of Newcastle, Australia}

\begin{document}
\maketitle

\abstract

We present an analytical upper bound on the number of required vehicles for vehicle routing problems with split deliveries and any number of capacitated depots. We show that a fleet size greater than the proposed bound is not achievable based on a set of common assumptions. This property of the upper bound is proved through a dynamic programming approach. We also discuss the validity of the bound for a wide variety of routing problems with or without split deliveries.

\noindent\keywords Split delivery; Number of vehicles; Multiple depots; Dynamic programming

\section{Introduction}
\label{Sec_Introduction}

Vehicle routing problems constitute an important class of combinatorial optimization problems. We develop an analytical upper bound on the number of required vehicles for split delivery routing problems with any number of capacitated depots. The upper bound is based on the following four assumptions:
\begin{itemize}
	\setlength\itemsep{0.1cm}
	\item \texttt{A1}:~Each vehicle delivers to exactly one depot.
	\item \texttt{A2}:~Vehicles are homogeneous.
	\item \texttt{A3}:~The aggregate load for each pair of vehicles exceeds the capacity of one vehicle.
	\item \texttt{A4}:~All parameters as well as the load for each vehicle are integers.
\end{itemize}

Without loss of generality, we assume that goods are collected from the demand points (customers with a pickup demand) and delivered to a number of depots. Note that there is an equivalence between pure distribution and pure collection since we can reverse the routes so that collection becomes distribution and vice versa \citep{toth2014}.

Assumption \texttt{A1} is quite common as many routing problems assume that each route begins and ends at the same depot. It is also valid when vehicles start and end their routes at different depots, but each of them makes a delivery to exactly one depot. This includes problems where a vehicle can visit any number of nodes (depots or demand points) to pick up goods and deliver to exactly one depot. Assumption \texttt{A2} is valid for a large proportion of the routing problems where the fleet of vehicles is homogeneous. In heterogeneous routing problems, we may have different types of vehicles that can differ in capacity, variable and fixed costs, speeds, and the customers that they can access \citep{toth2014,Kocc2016}. More precisely, the set of vehicles is partitioned into a number of subsets of homogeneous vehicles each corresponding to a vehicle type. In these problems, we are interested to know the number of required vehicles for each vehicle type. Given that vehicles belonging to a subset are homogeneous, we may employ the proposed upper bound for each subset of vehicles separately by assigning the total accessible demands to them.

Assumption \texttt{A3} is typically valid for routing problems where the cost/time or distance matrix satisfies the triangle inequality. If the total load for any pair of vehicles is less than or equal to the capacity of each vehicle, then we can simply combine the loads of the two vehicles and use only one vehicle to deliver them. The triangle inequality implies that the new route is not longer than the sum of the lengths of the previous two routes. Note that in the case where the triangle inequality does not hold, we sometimes can obtain an equivalent instance of the routing problem that satisfies the triangle inequality. This can be done by simply replacing the actual distance between each pair of nodes with the length of a shortest path connecting them (i.e., finding the metric closure of the network).

Assumption \texttt{A4} is valid for many real-world routing problems where a fractional number representing demand or capacity can be rounded to an integer. When parameters of the problem are all integers, the underlying routing problem often has an optimal solution in which the load for each vehicle is an integer. For example, \cite{Archetti2006tabu} consider a split delivery routing problem and show that when demands and capacity of each vehicle are integers, then there exists an optimal solution in which vehicle loads are all integers. Furthermore, when split deliveries are not allowed and demands are all integers, then the vehicle loads in every feasible solution are all integers.

We formulate two maximization problems based on Assumptions \texttt{A1}-\texttt{A4} and obtain a closed-form solution for each of them. The optimal objective value of these optimization problems give a tight upper bound for the fleet size in single-depot and multiple-depot routing problems, respectively. We refer to such a bound as ``tight'' in the sense that a fleet size greater than the bound is not achievable, although the bound itself may also not be achievable in some routing problems.

The second optimization problem represents the general case where we have any number of depots. We use a dynamic programming approach to obtain an analytical solution for this problem. The upper bound is useful as it can be employed to design solution algorithms for relevant routing problems. Furthermore, vehicle-flow (often three-index) formulations of routing problems require the maximum number of vehicles a priori \citep{toth2014}. The number of decision variables and constraints of these formulations are greatly affected by the upper bound on the number of vehicles used in the optimization. Consequently, it is important to specify a good upper bound to reduce the computational time. \cite{Chandran2008} solve the linear programming relaxation of a mixed integer linear program to obtain a good upper bound on the number of vehicles. This upper bound is utilized to reduce the number of decision variables of a routing problem in a preprocessing approach. \cite{Archetti2011} present an upper bound on the number of vehicles for a routing problem and use this upper bound in a column generation approach to solve the problem to optimality.

Our proposed upper bound can be computed very efficiently which makes it desirable for implementation within exact methods or as a preprocessing step. Although we focus on the routing problems with split deliveries, the proposed bound is also valid for the problems where split deliveries are not allowed. This can be realized by considering the upper bound as the optimal objective value of a maximization problem with the objective function being the number of vehicles used. Allowing split deliveries can be seen as a relaxation of this maximization problem which gives a valid upper bound for the corresponding routing problem. Furthermore, we point out that our proposed upper bound is valid for routing problems in which Assumption~\texttt{A1} does not hold. The reason is that removing Assumption~\texttt{A1} provides more degrees of freedom for the corresponding routing problem (removing a restriction on the number of deliveries that a vehicle can make). However, our proposed upper bound might not be the best possible bound that can be obtained for such routing problems based on our assumptions.

Fleet size (or often the number of routes) is a defining characteristic of a vehicle routing problem. Various routing problems have been defined and solved in the literature over the past six decades. In these works, the fleet size has almost always been fixed to the minimum possible number of routes \citep{uchoa2017new}. In other cases, only a trivial upper bound on the fleet size has been proposed.

In general, there is no relation between the number of vehicles used in a feasible solution and the quality of that solution. As argued by \cite{uchoa2017new}, the number of routes should not be fixed a priori as this is also the case in the original work of \cite{Dantzig1959}. To the best of our knowledge, the present paper is the first attempt to propose a non-trivial upper bound that is readily applicable to a wide variety of routing problems.

Another significance of our work is the development of closed-form upper bounds. It is possible to reformulate the non-linear formulations presented in this paper as a mixed integer linear program by using two-dimensional vehicle-index variables. However, the closed-form bounds developed in the paper (the optimal objective value of such linear formulations) can be computed in $\mathcal{O}(n\log n)$, where $n$ is the number of nodes. This improvement is particularly advantageous for very large instances of vehicle routing problems with thousands of demand points \citep[see e.g.,][]{campbell2004decomposition,qi2012spatiotemporal}. The bounds could also prove useful in more recent solution approaches to node routing problems such as quantum computing \citep{feld2019hybrid}.

The rest of the paper is organized as follows. In Section~\ref{Sec_Notation} we review the notations and definitions that are used throughout the paper. The upper bounds on the number of vehicles for routing problems with single and multiple depots are developed in Section~\ref{Sec_SDVRP} and Section~\ref{Sec_Multi}, respectively. Finally, we give some concluding remarks in Section~\ref{Conclusion}.

\section{Notation}
\label{Sec_Notation}

We denote the set of all non-negative integers by $\mathbb{Z}_0$, the set of all positive integers by $\mathbb{Z}_+$ and the set of all integers by $\mathbb{Z}$. For a given set $\mathcal{S}$, we use $|\mathcal{S}|$ to denote its cardinality. The floor and ceiling functions are denoted by $\floor{\cdot}$ and $\ceil{\cdot}$, respectively. We sometimes take advantage of the inequality $r \le \ceil{r} < r+1$ for a real number $r$, and refer to it as the \textit{ceiling property}.

We denote the set of demand points by $\mathcal{I}$ and the set of potential locations for facilities (depots) by $\mathcal{J}$. We often use $n$ to denote the number of potential facilities, that is, $n=|\mathcal{J}|$. Demand point $i \in \mathcal{I}$ has a demand $d_i \in \mathbb{Z}_0$ that must be satisfied. The total demand is denoted by $\varDelta = \sum_{i \in \mathcal{I}} d_i$. We assume that transportation of demand from demand points to the depots is carried out by a fleet of homogeneous vehicles each with capacity $q \in \mathbb{Z}_+$. The ratio $\frac{q+1}{2}$ is frequently used in the paper and therefore we use $Q$ to denote this ratio, that is,\[
Q\coloneqq \frac{q+1}{2}\cdot    
\] The capacity of depot $i\in\mathcal{J}$ is denoted by $c_i \in \mathbb{Z}_+$. We assume, without loss of generality, that $c_1 \ge \dots \ge c_n$.

\section{Single depot}
\label{Sec_SDVRP}

When there is only one depot (i.e., $n = 1$ and $\mathcal{I} \cap \mathcal{J} = \emptyset$), we may consider $\sum_{i \in \mathcal{I}} \cdiv{d_i}{q}$ as a trivial upper bound on the number of vehicles \citep{lee2006,Archetti2008}. The upper bounds $\cdiv{2\varDelta}{q}$ and $2\cdiv{\varDelta}{q}$ have also been proposed by \cite{labbe1991} and \cite{Archetti2011}, respectively. In this section, we develop an optimization problem to obtain a tight upper bound based on Assumptions \texttt{A1}-\texttt{A4}. Let $m$ be the number of required vehicles and $v_i$ denote the load of vehicle $i \in \{1,\dots,m\}$. We are interested to know the optimal objective value of the following single-depot formulation (\texttt{SDF}):
\begin{align}
\max ~~ &m && \label{Single_depot_obj}\\
\text{ s.t.~~}\ & \sum_{i=1}^{m} v_{i} = \varDelta && \label{Single_depot_1}\\
& v_{i} + v_{j} \ge q+1 &&   i,j \in \{1,\dots,m\}  \text{  with  } i \ne j \label{Single_depot_2}\\
& v_{i} \in \{1,\dots,q\} &&   i \in \{1,\dots,m\} \label{Single_depot_3}\\
& m \in\mathbb{Z}_0. && \label{Single_depot_4}
\end{align}
The objective of the \texttt{SDF} is to maximize the number of required vehicles, $m$. Constraint~\eqref{Single_depot_1} states that the total demand $\varDelta$ must be delivered to the depot by $m$ vehicles. Constraints~\eqref{Single_depot_2} are due to Assumptions \texttt{A3} and \texttt{A4}. Constraints \eqref{Single_depot_3}~and~\eqref{Single_depot_4} specify domains of the decision variables.
\begin{proposition}
	\label{Proposition_Pi}
	For $\varDelta \in \mathbb{Z}_0$ and $q \in \mathbb{Z}_+$, the function $\pi \colon \mathbb{Z}_0 \to \mathbb{Z}_0$ defined by
	\[ \pi(\varDelta) \coloneqq \begin{cases}
	
	\bceil {\dfrac{\varDelta}{q}} &   \varDelta \le q \\[8pt]
	\bcdiv{\varDelta-q}{\lceil Q \rceil } + 1 &   \varDelta > q,
	\end{cases} \]
	gives the optimal value of the \texttt{SDF}.
\end{proposition}

\begin{proof}
	It is clear that for $\varDelta = 0$ we do not need any vehicles. Due to Constraint~\eqref{Single_depot_2}, we can use only one vehicle for the case $1\le\varDelta \le q$, that is, $\pi(\varDelta)=1$. Therefore, let us suppose that $\varDelta > q$. Without loss of generality, we assume that vehicles are sorted such that $v_1 \ge \dots \ge v_m$. This implies that\[
	\varDelta - v_m = v_1 + \dots + v_{m-1} \ge (m-1)v_{m-1},
	\]
	and therefore $m \le \frac{\varDelta-v_m}{ v_{m-1} } + 1$.
	Given that $v_m+v_{m-1} \ge q+1$ and $v_m \le v_{m-1}$, we obtain $v_{m-1} \ge Q$. Since $m$ and $v_{m-1}$ take integer numbers, we can write $v_{m-1} \ge \left\lceil Q \right\rceil$ and
	\begin{align*}
	m & \le  \bfdiv{\varDelta-v_m}{ v_{m-1} } + 1 = \bfdiv{\varDelta-(v_m+v_{m-1})}{v_{m-1} } + 2. 
	\end{align*}
	The last inequality gives an upper bound on $m$. This upper bound is a function of $v_{m-1}$ and $v_m$ subject to $v_{m-1} \ge \left\lceil Q \right\rceil$, $v_{m-1}+v_m\ge q+1$ and $v_m \le v_{m-1}$. It is clear that the maximum of this upper bound is achieved if we substitute $v_m+v_{m-1}$ and $v_{m-1}$ with their lower bounds, that is, $v_m+v_{m-1}=q+1$ and $v_{m-1}= \lceil Q\rceil $. The constraint $v_m \le v_{m-1}$ is also satisfied. Thus, 
	\begin{align}
	m & \le \bfdiv{\varDelta-q-1}{ \left\lceil Q \right\rceil } + 2 = \bcdiv{\varDelta-q}{\lceil Q \rceil } + 1. \label{proof_single_depot_pi_final}
	\end{align}
	The last equality holds because $\fdiv{\alpha}{\beta}=\cdiv{\alpha+1}{\beta}-1$ for $\alpha\in\mathbb{Z}$ and $\beta\in\mathbb{Z}_+$ \citep{graham1994}. Therefore, any feasible solution of the \texttt{SDF} which employs \[m^\star \coloneqq \bcdiv{\varDelta-q}{\lceil Q \rceil } + 1\] vehicles is optimal. Let us define the solution
	\[ v_i^\star \coloneqq \begin{cases}
	\left\lceil Q \right\rceil &   i < m^\star \\
	\varDelta-\lceil Q\rceil \bcdiv{\varDelta-q}{\lceil Q\rceil} &   i = m^\star.
	\end{cases} \]
	It can be easily verified that Constraints \eqref{Single_depot_1} and \eqref{Single_depot_4} are satisfied by this solution. Now we use the ceiling property to show that Constraints \eqref{Single_depot_2} and \eqref{Single_depot_3} are also satisfied. Constraint \eqref{Single_depot_2} is satisfied since $v_i^\star + v_j^\star = $
	\[\begin{cases}
	\left\lceil Q \right\rceil + \left\lceil Q \right\rceil \ge q+1 &   i,j < m^\star \\[7pt]
	\varDelta-\lceil Q\rceil \left(\bcdiv{\varDelta-q}{\lceil Q\rceil}-1\right) > \varDelta-\lceil Q\rceil \left(\dfrac{\varDelta-q}{\lceil Q\rceil}\right) = q &   i < j = m^\star
	\end{cases} \]
	For Constraint \eqref{Single_depot_3}, we check the inequality $v^\star_i \le q$ as follows:
	\[ v_i^\star = \begin{cases}
	\left\lceil Q \right\rceil \le q &   i < m^\star \\
	\varDelta-\lceil Q\rceil \bcdiv{\varDelta-q}{\lceil Q\rceil} \le \varDelta-\lceil Q\rceil \left( \dfrac{\varDelta-q}{\lceil Q\rceil}\right)=q &   i = m^\star
	\end{cases} \]
	The inequality $v^\star_i \ge 1$ is satisfied since
	\[ v_i^\star = \begin{cases}
	\left\lceil Q \right\rceil \ge 1 &   i < m^\star \\
	\varDelta-\lceil Q\rceil \bcdiv{\varDelta-q}{\lceil Q\rceil} > \varDelta-\lceil Q\rceil \left( \dfrac{\varDelta-q}{\lceil Q\rceil} + 1\right) \ge 0 &   i = m^\star
	\end{cases} \]
	Hence Constraint \eqref{Single_depot_3} is also satisfied.
\end{proof}

\section{Multiple depots}
\label{Sec_Multi}

In this section, we assume that $n \ge 1$ and $\mathcal{I} \cap \mathcal{J} = \emptyset$. We introduce an optimization problem whose optimal solution provides a tight upper bound on the number of vehicles. Let the decision variable $x_i$ represent the quantity of demand delivered to depot $i\in \mathcal{J}$. The multi-depot formulation (\texttt{MDF}) can then be expressed as follows:
\begin{align}
\max ~~ & \sum_{i=1}^{n} \pi(x_i) && \label{Multi_depot_obj_0}\\
\text{ s.t.~~}\ & \sum_{i=1}^{n} x_{i} \le \varDelta && \label{Multi_depot_1_0}\\
& x_{i} \in \{0,\dots,c_i\} &&   i \in \{1,\dots,n\}. \label{Multi_depot_2_0}
\end{align}
Constraint~\eqref{Multi_depot_1_0} states that the total quantity delivered by vehicles must not exceed the total demand,~$\varDelta$. Constraints~\eqref{Multi_depot_2_0} specify domains of the decision variables. Observe that the decision variable $x_i$ can take the value zero. This means that it is not required that every depot has positive demand delivered. Consequently, the bound is valid whether the locations of the depots are predetermined or not.

The objective of the \texttt{MDF} is to maximize the number of required vehicles. The function $\pi(x_i)$ represents the maximum number of vehicles required to deliver $x_i$ units of demand to depot $i \in \mathcal{J}$. Because each vehicle can deliver to exactly one depot (by Assumption~\texttt{A1}), the objective function $\sum_{i=1}^{n} \pi(x_i)$ properly represents the total number of required vehicles. When there is only one depot, we know that the function $\pi$ introduced in Section~\ref{Sec_SDVRP} gives a tight upper bound on the number of vehicles. Therefore, the optimal value of the \texttt{MDF} gives a tight bound on the number of vehicles when $n$ depots are available. In the absence of Assumption~\texttt{A1}, the resulting optimal value of the \texttt{MDF} is still a valid upper bound. However, it might not be a tight upper bound for every class of instances.

We solve the \texttt{MDF} using a dynamic programming approach. Let $V^\star_j(\delta_j)$ be the maximum number of vehicles required to deliver $\delta_j \in \{0,\dots,\Delta\}$ units of demand to depots $1,\dots,j$. In other words,
\begin{align}
V^\star_j(\delta_j) \coloneqq \max_{x_1,\dots,x_j} ~~ & \sum_{i=1}^{j} \pi(x_i) && \label{Multi_depot_obj_DP}\\
\text{ s.t.~~}\ & \sum_{i=1}^{j} x_{i} \le \delta_j && \label{Multi_depot_1_DP}\\
& x_{i} \in\{0,\dots,c_i\} &&   i \in \{1,\dots,j\}. \label{Multi_depot_2_DP}
\end{align}

\noindent We refer to this formulation as the Dynamic Programming Formulation (\texttt{DPF$_j$}). We therefore obtain the optimality equation
\begin{align} 
V^\star_j(\delta_j) = \begin{cases}
\displaystyle\max_{x_j\in \{0,\dots,\min(c_j,\delta_j)\}} \pi(x_j)=\pi(\min(c_j,\delta_j)) &   j = 1 \\[10pt]
\displaystyle\max_{x_j\in \{0,\dots,\min(c_j,\delta_j)\}} V_j(x_j,\delta_j) &   j > 1, \label{Opt_Eq_UB}
\end{cases}
\end{align}
\noindent in which $V_j(x_j,\delta_j) \coloneqq \pi(x_j) + V^\star_{j-1}(\delta_j-x_j)$. We can employ Equation \eqref{Opt_Eq_UB} to obtain an optimal policy for the \texttt{MDF}. Let us begin by introducing the function $\theta \colon \mathbb{Z}_0 \to \mathbb{Z}_0$ defined by
\begin{align}
\theta(\alpha) \coloneqq \begin{cases}
~\pi(\alpha) &   \alpha \le q\\
(\pi(\alpha)-1)\left\lceil Q \right\rceil+\left\lfloor Q \right\rfloor &   \alpha > q. \label{Def_Theta}
\end{cases}
\end{align}
This function has some interesting properties that allow us to obtain an analytical solution.

\begin{lemma}
	\label{Lemma_dual}
	For each $\alpha \in \mathbb{Z}_0$ we have $\theta(\alpha) \le \alpha$ and $\pi(\theta(\alpha))=\pi(\alpha)$.
\end{lemma}

\begin{proof}
	We know that $\pi(0)=0$ and $\pi(\alpha)=1$ for $1 \le \alpha\le q$. Therefore, it can be verified that the statement is true for $\alpha \le q$. For $\alpha > q$ we have $\pi(\alpha) = \bcdiv{\alpha-q}{\left\lceil Q \right\rceil}	+1 \ge 2$. Thus,
	\begin{align*}
	\theta(\alpha) &= \bcdiv{\alpha-q}{\left\lceil Q \right\rceil}	\left\lceil Q \right\rceil+\left\lfloor Q \right\rfloor \ge \left\lceil Q \right\rceil+\left\lfloor Q \right\rfloor = q+1.
	\end{align*}
	Furthermore, the ceiling property implies that
	\[
	\bcdiv{\alpha-q}{\left\lceil Q \right\rceil}	<  \frac{\alpha-q}{\left\lceil Q \right\rceil}  + 1,
	\]
	and therefore
	\[
	\theta(\alpha) < \alpha - q + \left\lceil Q \right\rceil+\left\lfloor Q \right\rfloor = \alpha+1.
	\]
	As a result, $q<\theta(\alpha)\le\alpha$. Since $\theta(\alpha)>q$ we can write that \[\pi(\theta(\alpha)) = \bcdiv{\theta(\alpha)-q}{\left\lceil Q \right\rceil}	+1.\]Now by considering the definition \eqref{Def_Theta} we obtain
	\begin{align*}
	\pi(\theta(\alpha)) &= (\pi(\alpha)-1) + \bcdiv{\left\lfloor Q \right\rfloor-q}{\left\lceil Q \right\rceil} + 1 \\
	& = (\pi(\alpha)-1) + \bcdiv{\left\lfloor Q \right\rfloor+\left\lceil Q \right\rceil-q}{\left\lceil Q \right\rceil}\\
	& = (\pi(\alpha)-1) + \bcdiv{1}{\left\lceil Q \right\rceil} =  \pi(\alpha).\qedhere
	\end{align*}
\end{proof}

\begin{lemma}\label{lemma_profit_function_revised}
	For each $\alpha,\beta \in \mathbb{Z}_+$ with $q< \alpha <  \beta-q$ we have
	\begin{align*}
	\medskip \pi(\alpha)+\pi(\beta-\alpha) & \le 1+\pi\left(\beta-\bfloor{Q}\right)\,.
	\end{align*}
\end{lemma}
\begin{proof}
	From Lemma~\ref{Lemma_dual}, we know that $\theta(\alpha)\le \alpha$. As $\beta-\alpha>q$, we obtain $\beta-\theta(\alpha) \ge \beta-\alpha >q.$
	Now by considering that $\alpha>q$ we can write
	\begin{align*}
	\medskip \pi(\alpha)+\pi(\beta-\alpha) & \le \pi(\alpha)+\pi(\beta-\theta(\alpha)) \\
	\medskip  & = \pi(\alpha)+\pi\left(\beta-(\pi(\alpha)-1)\bceil{Q}-\bfloor{Q}\right) \\
	\medskip  & = \pi(\alpha) - (\pi(\alpha)-1) + \bcdiv{\beta-\floor{Q}-q}{\ceil{Q}} + 1 \\
	\medskip  & = 1 + \pi\left(\beta-\bfloor{Q}\right)\,. 	
	\end{align*}
	Note that $\beta >\alpha + q > 2q$ which follows that $\beta-\bfloor{Q} > q$. This enabled us to write the last equality.
\end{proof}

\begin{lemma}
	\label{lemma_basic_property}
	For $\alpha \in\mathbb{Z}_0$, we have $\pi(\beta+1)-\pi(\beta) \le 1$.
\end{lemma}

\begin{proof}
	The result follows from the ceiling property.
\end{proof}

\begin{lemma}
	\label{lemma_profit_function}
	For each $\alpha\in\mathbb{Z}_0$ and $\beta \in \mathbb{Z}_+$ with $\alpha \le \beta$ we have \[\pi(\alpha)+\pi(\beta-\alpha) \le 1 + \pi(\beta-1).\]
\end{lemma}

\begin{proof}
	From Lemma~\ref{lemma_basic_property} we know that $\pi(\beta)-\pi(\beta-1) \le 1$. Thus, the result follows for $\alpha=0$. Now consider the case $1 \le \alpha\le q$ in which $\pi(\alpha)=1$. Since $\pi(\cdot)$ is an increasing function and $\beta-\alpha \le \beta-1$, the result follows. Finally, we suppose that $q<\alpha\le\beta$. Since the case $\alpha = \beta$ is equivalent to the case $\alpha=0$, we assume that $q < \alpha \le \beta-1$.\\
	If $\beta-\alpha \le q$, then we obtain
	\begin{align*}
	\pi(\alpha)+\pi(\beta-\alpha) = \pi(\alpha)+1 \le \pi(\beta-1) +1.
	\end{align*}
	For $\beta-\alpha > q$, we can utilize Lemma~\ref{lemma_profit_function_revised} to obtain
	\begin{align*}
	\pi(\alpha)+\pi(\beta-\alpha) & \le 1+\pi\left(\beta-\bfloor{Q}\right) \le 1 + \pi(\beta-1). \qedhere
	\end{align*}
\end{proof}

\begin{lemma}
	\label{Lemma_pi_x_decreasing}
	For $\alpha \in\mathbb{Z}_0$, we have $\pi(\alpha) \le \alpha$.
\end{lemma}

\begin{proof}
	Given that $\alpha \ge 0$, we just need to show that the function $f(\alpha) \coloneqq \pi(\alpha)-\alpha$ is decreasing. We have \[
	f(\alpha+1)-f(\alpha) = \pi(\alpha+1)-\pi(\alpha) - 1 \le 0,
	\]
	which is true due to Lemma~\ref{lemma_basic_property}. 
\end{proof}

\begin{lemma}
	\label{Lemma_trivial_solution}
	If $\delta_j \le j$, then
	\[ x^\star_i \coloneqq \begin{cases}
	1, &   i = 1, \dots, \delta_j \\
	0, &   i = \delta_j + 1, \dots, j
	\end{cases} \]
	is an optimal solution for \texttt{DPF$_j$}. Furthermore, $V^\star_j(\delta_j)=\delta_j$.
\end{lemma}
\begin{proof}
	It is clear that the given solution is feasible for \texttt{DPF$_j$} and $\sum_{i=1}^{j} \pi(x^\star_i) = \delta_j$. The proof is complete if we show that $\delta_j$ is an upper bound for $V^\star_j(\delta_j)$. According to Lemma~\ref{Lemma_pi_x_decreasing},\[
	\sum_{i=1}^{j} \pi(x_i) \le \sum_{i=1}^{j} x_i \le \delta_j,
	\]which states that $\delta_j$ is an upper bound for $V^\star_j(\delta_j)$.
\end{proof}

\begin{lemma}
	\label{Lemma_x_i_1}
	Consider an instance of \texttt{DPF$_j$} with $\delta_j \ge j$. There exists an optimal solution for this instance in which $x_i \ge 1$ for $i=1,\dots,j$.
\end{lemma}

\begin{proof}
	Let $(x^\star_1,\dots,x_j^\star)$ be an optimal solution of \texttt{DPF$_j$}. If $x^\star_i \ge 1$ for $i = 1,\dots,j$ then the proof is complete. Therefore, suppose that $x^\star_l=0$ for some $l\in\{1,\dots,j\}$. Since the solution is optimal and $\delta_j \ge j$, there must exist some $k\in\{1,\dots,j\}$ with $x^\star_k \ge 2$. Now we construct a new feasible solution by setting $x_l^{new}=1$ and $x_k^{new}=x^\star_k-1$. According to Lemma~\ref{lemma_basic_property},\[
	\pi(x^\star_l) + \pi(x^\star_k) \le \pi(x^{new}_l) + \pi(x^{new}_k),
	\] 
	implying that the new feasible solution is also optimal. We can repeat this procedure until we have an optimal solution in which $x_i \ge 1$ for $i=1,\dots,j$.
\end{proof}

\begin{theorem}
	\label{Theorem_discrete}
	For $\delta_j \ge j$, the optimal value of $\texttt{DPF}_j$ is given by\[
	V_j^\star(\delta_j) = j-1+\pi(\min(\lambda_{\ell_j},c_{\ell_j})) + \sum_{i=1}^{\ell_j-1} (\pi(c_i)-1),\]
	where $\lambda_{\ell_j} \coloneqq (\delta_j-j+1)-\sum_{r=1}^{\ell_j-1}(\theta(c_r)-1)\,$,
	and $\ell_j \in \{1,\dots,j\}$ is the largest integer such that $\lambda_{\ell_j} \ge 1$\,. 
\end{theorem}
\begin{proof}
	It is enough to show that the policy 
	\begin{align}
	x_i^\star & \coloneqq \begin{cases} 
	\theta(c_i)&  i=1,\dots,{\ell_j}-1 \\
	\theta(\min(\lambda_{\ell_j}\,,c_{\ell_j})) &   i={\ell_j} \\
	1 &   i={\ell_j}+1,\dots,j,
	\end{cases}\label{OptSol}
	\end{align}
	is optimal since it is a feasible policy that yields $V_j^\star(\delta_j)$. From Lemma~\ref{Lemma_dual} we know that $\theta(c_i) \le c_i$ for $i=1,\dots,j$, implying that the given policy \eqref{OptSol} is feasible. We give a proof by induction on~$j$ to show that it is also an optimal policy. For $j=1$ the given policy states that $\ell_1 = 1$  and $x^\star_1 = \theta(\min(\delta_1,c_1))$. In addition, the optimality equation \eqref{Opt_Eq_UB} states that $V^\star_1(\delta_1)=\pi(\min(\delta_1,c_1))$. From Lemma~\ref{Lemma_dual} we know that $\pi(\min(\delta_1,c_1)) = \pi (\theta(\min(\delta_1,c_1)))$. Therefore, $x^\star_1 = \theta(\min(\delta_1,c_1))$ is an optimal policy for \texttt{DPF$_1$}, implying that the assertion is true for $j=1$. Next we show that the assertion holds for $j=2$, that is, the given policy \eqref{OptSol} is optimal for $\texttt{DPF}_2$. We will refer to this case later. We know that
	\begin{align*}
	\medskip V_2^\star(\delta_2) & = \max_{x_2 \in \{0,\dots,\min(\delta_2,c_2) \} }V_2(x_2,\delta_2).
	\end{align*}
	Here, we consider two possible cases for $\delta_2$ and show that in both of them the given policy \eqref{OptSol} is optimal.
	
	\noindent\textbf{Case 1: $\bm{\delta_2 \le \theta(c_1)}$\,:} In this case, Equation \eqref{OptSol} prescribes that $\ell_2 = 1$, $x_2^\star=1$ and $x_1^\star=\theta(\min(\delta_2-1,c_1))$. Given that $\delta_2 \le \theta(c_1)$, Lemma~\ref{Lemma_dual} implies that $\delta_2 \le c_1$ and therefore $x_1^\star=\theta(\delta_2-1)$. Now, by utilizing Lemma~\ref{lemma_profit_function} and Lemma~\ref{Lemma_dual}, we have
	\begin{align*}
	\medskip V_2(x_2,\delta_2) & =\pi(x_2)+\pi(\min (c_1,\delta_2-x_2)) \\
	\medskip  & = \pi(x_2)+\pi(\delta_2-x_2) \\
	\medskip  & \le 1+\pi(\delta_2-1) = 1+\pi(\theta(\delta_2-1))\,,
	\end{align*} 
	for any $x_2\in\{0,\dots,\min(\delta_2,c_2)\}$\,. Thus, $(x_1^\star,x_2^\star)$ is an optimal policy for $\texttt{DPF}_2$. 
	
	\noindent\textbf{Case 2: $\bm{\delta_2 > \theta(c_1)}$\,:} In this case, Equation \eqref{OptSol} prescribes that $\ell_2 = 2$, $x_1^\star=\theta(c_1)$ and $x_2^\star=\theta(\min(\delta_2-\theta(c_1),c_2))$\,. There are three possible categories to choose $x_2$\,:
	
	\textbf{(i) Choose $\bm{x_2 \le q}$:} We have
	\begin{align*}
	\medskip V_2(x_2,\delta_2) & = \pi(x_2)+\pi(\min(\delta_2-x_2,c_1)) \\ 
	\medskip  & \le \pi(q)+\pi(c_1) \\
	\medskip  & =1+\pi(x_1^\star) \le \pi(x_2^\star) + \pi(x_1^\star).
	\end{align*} 
	
	Hence, there is no policy better than $(x_1^\star,x_2^\star)$ when $x_2 \le q$\,.
	
	\textbf{(ii) Choose $\bm{q< x_2 <\delta_2 - q }$:} Since $x_2 \le c_2 \le c_1$, the choice of $x_2 > q$ can be feasible if $c_2 > q$ and $c_1 > q$. By utilizing Lemma \ref{lemma_profit_function_revised}, we have
	\begin{align*}
	\medskip V_2(x_2,\delta_2) & = \pi(x_2)+\pi(\min(\delta_2-x_2,c_1)) \\
	\medskip  & \le \pi(x_2)+\pi(\delta_2-x_2) \\
	\medskip  & \le 1+\pi\left(\delta_2-\bfloor{Q}\right).
	\end{align*}
	
	Now it is enough to show that the right hand side of the last inequality is less than or equal to $\pi(x_1^\star)+\pi(x_2^\star)$. For this purpose, we first consider the case $\delta_2 - \theta(c_1) \le q$.	In this case, we obtain $x_2^\star=\theta(\delta_2-\theta(c_1))$ with $\pi(x_2^\star) = 1$ because $c_2 > q$ and $\delta_2 > \theta(c_1)$. Therefore, we have $\pi(x_1^\star)+\pi(x_2^\star) = 1 + \pi(c_1)$.
	
	Given that $\delta_2 - \theta(c_1) \le q$, we conclude that
	\begin{align*}
	\pi\left(\delta_2-\bfloor{Q}\right) \le \pi \left( \theta(c_1) + q - \bfloor{Q} \right).
	\end{align*}
	
	So, it is enough to show that the right hand side of the last inequality is equal to $\pi(c_1)$. First note that $c_1 >q$ which implies that $\theta(c_1) > q$ (see the proof of Lemma~\ref{Lemma_dual}). On the other hand, $q - \floor{Q} \ge 0$ for any $q \in \mathbb{Z}_+$. Therefore we have
	\begin{align*}
	\pi \left( \theta(c_1) + q - \bfloor{Q} \right) &= \pi \left( (\pi(c_1) -1) \bceil{Q} + q \right) \\
	& = (\pi(c_1) -1) + 1 = \pi(c_1).
	\end{align*}
	
	Next we consider the case $\delta_2 - \theta(c_1) > q$, in which we have
	\begin{align*}
	\pi(x_1^\star)+\pi(x_2^\star) = \pi(c_1)+\pi(\min(\delta_2-\theta(c_1),c_2)).
	\end{align*}
	If $\delta_2-\theta(c_1) > c_2$ then $\pi(x_1^\star)+\pi(x_2^\star) = \pi(c_1)+\pi(c_2)$ which is clearly an upper bound for $V_2(x_2,\delta_2)$. This implies that there is no policy better than $(x_1^\star,x_2^\star)$ in this case. Therefore, we assume that $q < \delta_2-\theta(c_1) \le c_2$ which follows that
	\begin{align*}
	\medskip \pi(x_1^\star)+\pi(x_2^\star) & = \pi(c_1)+\pi(\delta_2-\theta(c_1)) \\
	\medskip  & = \pi(c_1)+\pi\left(\delta_2-(\pi(c_1)-1)\bceil{Q}-\bfloor{Q}\right) \\
	\medskip  & = \pi(c_1)-(\pi(c_1)-1)+ \bcdiv{\delta_2 - \floor{Q} - q}{\ceil{Q}}+1\\
	\medskip  & = 1+\pi\left(\delta_2-\bfloor{Q}\right)\,.
	\end{align*}
	
	\noindent Note that $\delta_2 > \theta(c_1) + q > 2q$ which follows that $\delta_2-\floor{Q} > q$.
	
	\textbf{(iii) Choose $\bm{x_2 \ge \delta_2-q}$:} In this case $\delta_2-x_2 \le q$ and because $x_2 \le c_2 \le c_1$ we can write
	\begin{align*}
	\medskip V_2(x_2,\delta_2) & =\pi(x_2)+\pi(\min (c_1,\delta_2-x_2)) \\
	\medskip  & \le \pi(x_2)+\pi(\delta_2-x_2) \le \pi(x_2)+1 \\
	\medskip  & \le \pi(c_1)+1 \le  \pi(x_1^\star) + \pi(x_2^\star).
	\end{align*} 
	
	\noindent Now, let us assume that the given policy \eqref{OptSol} is optimal for $\texttt{DPF}_k$, where $k\ge 2$. We show that it is also optimal for $\texttt{DPF}_{k+1}$. For $j=k+1$ we have
	\begin{align*}
	V^\star_{k+1}(\delta_{k+1}) &= \max_{x_{k+1}\in \{1,\dots,\min(c_{k+1},\delta_{k+1})\}} V_{k+1}(x_{k+1},\delta_{k+1}),
	\end{align*}
	in which $\delta_{k+1} \ge k+1$ and
	\begin{align*}
	V_{k+1}(x_{k+1},\delta_{k+1}) &= \pi(x_{k+1}) + V^\star_{k}(\delta_{k+1}-x_{k+1}).
	\end{align*}
	Note that we assume $x_{k+1} \ge 1$ which is valid due to Lemma~\ref{Lemma_x_i_1}. From the induction hypothesis we know that Equation \eqref{OptSol} can be used to obtain $V^\star_{k}(\delta_{k+1}-x_{k+1})$. Let $x_1^\star,\dots,x_k^\star$ be the resulting optimal policy, that is,
	\begin{align*}
	x_i^\star & = \begin{cases} 
	\theta(c_i)&  i=1,\dots,{\ell_k}-1 \\
	\theta(\min(\lambda_{\ell_k}\,,c_{\ell_k})) &   i={\ell_k} \\
	1 &   i={\ell_k}+1,\dots,k,
	\end{cases}
	\end{align*}
	where $\ell_k \in \{1,\dots,k\}$ is the largest integer such that \[\lambda_{\ell_k} = (\delta_{k+1}-x_{k+1}-k+1)-\sum_{r=1}^{\ell_k-1}(\theta(c_r)-1) \ge 1.\]
	Furthermore, let $\bm{z} = (z_1,\dots,z_{k+1})$ be the solution given by Equation \eqref{OptSol} to solve $\texttt{DPF}_{k+1}$. The proof is completed if we can show that\[
	V_{k+1}(x_{k+1},\delta_{k+1}) \le \sum_{i=1}^{k+1} \pi(z_i),
	\] 
	for any $x_{k+1}\in \{1,\dots,\min(c_{k+1},\delta_{k+1})\}$. In other words, it is enough to show that the solution $\bm{x}\coloneqq(x_1^\star,\dots,x_k^\star,x_{k+1})$ is not better than $\bm{z}$. To this end, we first define the solution $\bm{u}= (u_1,\dots,u_{k+1})$ as follows:  
	\[u_i \coloneqq \begin{cases}
	x_i^\star &   i=1,\dots,\ell_k \\
	x_{k+1} &   i={\ell_k}+1 \\
	1 &   i=\ell_k+2, \dots, k+1.
	\end{cases}\]  
	This solution is illustrated in Figure \ref{fig:w_illustration}. Since $c_{k+1} \le c_{{\ell_k}+1}$\,, $\bm{u}$ is feasible. Moreover, both $\bm{x}$ and $\bm{u}$ give the same objective value.	Next we construct the feasible solution $\bm{w} = (w_1,\dots,w_{k+1})$ based on $\bm{u}$ whose objective value is better than or equal to that of $\bm{u}$. To construct such a solution, we first set $w_i=u_i$ for $i=1,\dots,k+1$ except for $i={\ell_k},{\ell_k}+1$\,(see Figure \ref{fig:w_illustration}).
	\begin{figure}[h]
		\centering
		\begin{tikzpicture}
		
		\newcommand\XX{0.9}
		\newcommand\YY{0.5}
		
		\begin{scope}[every node/.style={circle,minimum size=25pt,inner sep=0pt,font=\scriptsize}]
		\node [] (A1) at (-1*\XX,3*\YY) {$k+1$};
		\node [] (A2) at (-2*\XX,3*\YY) {$k$};
		\node [] (A4) at (-3*\XX,3*\YY) {$\dots$};
		\node [] (A5) at (-4*\XX,3*\YY) {${\ell_k}+2$};
		\node [] (A6) at (-5*\XX,3*\YY) {${\ell_k}+1$};
		\node [] (A7) at (-6*\XX,3*\YY) {${\ell_k}$};
		\node [] (A31) at (-7*\XX,3*\YY) {$\ell_{k}-1$};
		\node [] (A8) at (-8*\XX,3*\YY) {$\dots$};
		\node [] (A9) at (-9*\XX,3*\YY) {$1$};
		
		\node [] (A10) at (-9.7*\XX,2.7*\YY) {};
		\node [] (A11) at (0,2.7*\YY) {};
		
		\node [] (A12) at (-9.7*\XX,2*\YY) {$\bm{x}\colon$};
		\node [] (A13) at (-1*\XX,2*\YY) {$x_{k+1}$};
		\node [] (A14) at (-2*\XX,2*\YY) {$1$};
		\node [] (A16) at (-3*\XX,2*\YY) {$\dots$};
		\node [] (A17) at (-4*\XX,2*\YY) {$1$};
		\node [] (A18) at (-5*\XX,2*\YY) {$1$};
		\node [] (A19) at (-6*\XX,2*\YY) {$x^\star_{{\ell_k}}$};
		\node [] (A32) at (-7*\XX,2*\YY) {$\theta(c_{\ell_{k}-1})$};
		\node [] (A20) at (-8*\XX,2*\YY) {$\dots$};
		\node [] (A21) at (-9*\XX,2*\YY) {$\theta(c_1)$};
		
		\node [] (A12) at (-9.7*\XX,1*\YY) {$\bm{u}\colon$};
		\node [] (A13) at (-1*\XX,1*\YY) {$1$};
		\node [] (A14) at (-2*\XX,1*\YY) {$1$};
		\node [] (A16) at (-3*\XX,1*\YY) {$\dots$};
		\node [] (A17) at (-4*\XX,1*\YY) {$1$};
		\node [] (A18) at (-5*\XX,1*\YY) {$x_{k+1}$};
		\node [] (A19) at (-6*\XX,1*\YY) {$x^\star_{{\ell_k}}$};
		\node [] (A33) at (-7*\XX,1*\YY) {$\theta(c_{\ell_{k}-1})$};
		\node [] (A20) at (-8*\XX,1*\YY) {$\dots$};
		\node [] (A21) at (-9*\XX,1*\YY) {$\theta(c_1)$};
		
		\node [] (A22) at (-9.7*\XX,0*\YY) {$\bm{w}\colon$};
		\node [] (A23) at (-1*\XX,0*\YY) {$1$};
		\node [] (A24) at (-2*\XX,0*\YY) {$1$};
		\node [] (A25) at (-3*\XX,0*\YY) {$\dots$};
		\node [] (A26) at (-4*\XX,0*\YY) {$1$};
		\node [] (A27) at (-5*\XX,0*\YY) {$w_{\ell_k+1}$};
		\node [] (A28) at (-6*\XX,0*\YY) {$w_{{\ell_k}}$};
		\node [] (A34) at (-7*\XX,0*\YY) {$\theta(c_{\ell_{k}-1})$};
		\node [] (A29) at (-8*\XX,0*\YY) {$\dots$};
		\node [] (A30) at (-9*\XX,0*\YY) {$\theta(c_1)$};
		
		\end{scope}
		
		\begin{scope}[>={Stealth[blue]},
		every node/.style={fill=white,circle}, 
		every edge/.style={draw=black,font=\scriptsize}]
		\path [-] (A10.east) edge  (A11.west);
		\end{scope}
		\end{tikzpicture}
		\caption{\small Comparing the solution vectors $\bm{x}$, $\bm{u}$ and $\bm{w}$} \label{fig:w_illustration}
	\end{figure}
	
	\noindent It is clear that \[
	\xi \coloneqq \delta_{k+1} - \sum_{i=1}^{\ell_k-1} u_i - \sum_{i=\ell_k+2}^{k+1} u_i = (\delta_{k+1} -k+1) - \sum_{i=1}^{\ell_k-1} (\theta(c_i)-1),
	\]
	is the remaining unmet demand that can be assigned to $w_{\ell_k}$ and $w_{\ell_k+1}$. Note that $\xi \ge x^\star_{\ell_k} + x_{k+1} \ge 2$. According to the induction hypothesis for $j=2$, Equation \eqref{OptSol} gives an optimal policy for assigning $\xi$ to $w_{\ell_k}$ and $w_{\ell_k+1}$ as follows:\[
	(w_{\ell_k},w_{\ell_k+1}) = \begin{cases}
	(\theta(\min(\xi -1,c_{\ell_k})),1)  &   2 \le \xi \le \theta(c_{\ell_k}) \\
	(\theta(c_{\ell_k}),\theta(\min(\xi-\theta(c_{\ell_k}),c_{\ell_{k}+1})))&   \xi \ge \theta(c_{\ell_k})+1.
	\end{cases}\]
	Given this policy, we can conclude that\[
	V_{k+1}(x_{k+1},\delta_{k+1}) = \sum_{i=1}^{k+1} \pi(u_i) \le \sum_{i=1}^{k+1} \pi(w_i).
	\] 
	Furthermore, $\bm{w}$ coincides with Equation \eqref{OptSol} for $\texttt{DPF}_{k+1}$ which implies that $\bm{w} = \bm{z}$. This completes the proof. 
\end{proof}

We summarize the results obtained in this section in Table~\ref{tabel_closed_form_discrete}. This table presents the optimal value of the \texttt{MDF} which is a tight upper bound on the fleet size based on our assumptions. Observe that the special case of $c_1 > \Delta$ is equivalent to the scenario where at least one depot (or facility in a location-routing problem) is uncapacitated.

\begin{table}[h]
	\centering
	\begin{threeparttable}
		\caption{The proposed upper bound on the number of vehicles}
		\begin{tabular}{@{} l l @{}}
			\toprule
			Case  & Optimal value\\
			\midrule
			$0 \le \varDelta \le n$  &   
			$\varDelta$\\
			
			$n \le \varDelta < n+q$  &   
			$n$\\
			
			$n+q \le \varDelta \le c_1$  &   
			$n-1 + \pi(\varDelta-n+1)$\\
			
			$\max(n+q,c_1) \le \varDelta$ ~~ &   
			$n-1+\pi(\min(\lambda_{\ell_n},c_{\ell_n})) + \sum_{i=1}^{\ell_n-1} (\pi(c_i)-1)\tnote{\textdagger}$ \\ \bottomrule
		\end{tabular}
		\begin{tablenotes}
			\item[\textdagger]\scriptsize $\ell_n \le n$ is the largest integer such that $\lambda_{\ell_n} = (\varDelta-n+1)-\sum_{k=1}^{\ell_n-1}(\theta(c_k)-1) \ge 1$.
		\end{tablenotes}
		\label{tabel_closed_form_discrete}%
	\end{threeparttable}
\end{table}

\section{Conclusion}
\label{Conclusion}

In this paper, we studied the maximum number of homogeneous vehicles in node routing problems with integer splits. We assumed that all parameters as well as the vehicle loads are integers. Furthermore, each vehicle can make a delivery to exactly one depot, and the aggregate load for each pair of vehicles exceeds the capacity of one vehicle. We first obtained a closed-form for the optimal objective value of a maximization problem. This gives a valid upper bound based on our assumptions when there is only one predetermined depot. We then used this upper bound to develop an optimization problem whose optimal value gives a tight upper bound for multiple capacitated depots. The desirable characteristics of the proposed upper bound is that it is optimal in the sense that it gives a tight bound based on our assumptions for any given instance. Furthermore, they can be computed very efficiently in $\mathcal{O}(n\log n)$.

A future research direction for this work is to obtain stronger bounds by extending the set of assumptions or constructing different optimization models. For example, we may consider a particular routing problem and obtain better upper bounds by considering more details such as routing costs. It would also be interesting to know how we can obtain better theoretical bounds on the number of vehicles if we have a heuristic bound on the optimal routing costs. This can be helpful particularly for exact methods that seek a global optimal solution. Since theoretical bounds can be computed very efficiently, they could also prove useful for solving very large routing instances. The upper bounds obtained in this paper are also valid for node routing problems in which split deliveries are not allowed and a vehicle can deliver to any number of depots. However, these upper bounds may not be tight for those routing problems. Therefore, another research direction is to develop better theoretical bounds for those routing problems. We believe that the present work could be a good starting point for such improvements and a baseline for evaluating them. It also opens a new avenue of research for various routing problems under various assumptions such as data uncertainty and multiple objectives.

\textbf{Acknowledgments.} We would like to thank Dr. Mojtaba Heydar for discussing this problem with us. This study was funded by the Australian Research Council (Grant ID: IC140100032).

\bibliographystyle{apalike}
\bibliography{BIBLIO}
\end{document}